%% file: indep.tex
\documentclass[11pt]{article}

\date{}

\input{Preamble}

\usepackage[T1]{fontenc}
\usepackage[utf8]{inputenc}
\usepackage{comment}
\usepackage{indentfirst}
\usepackage[top=1.25in,left=1.25in,right=1.25in]{geometry}

\author{\Large{Sébastien Gaspoz and Riccardo W. Maffucci}}

\title{\Large{\uppercase{\bf Independence numbers of polyhedral graphs}}}

\begin{document}
\maketitle
\nocite{dillencourt1996polyhedra}



\begin{abstract}
A polyhedral graph is a $3$-connected planar graph. We find the least possible order $p(k,a)$ of a polyhedral graph containing a $k$-independent set of size $a$ for all positive integers $k$ and $a$. In the case $k = 1$ and $a$ even, we prove that the extremal graphs are exactly the vertex-face (radial) graphs of maximal planar graphs.
\end{abstract}
{\bf Keywords:} Independent set, Extremal problem, Planar graph, $3$-polytope, Radial graph.
\\
{\bf MSC(2010):} 05C69, 05C35, 05C10, 52B05.

\section{Introduction}
Given a graph $G = (V,E)$, a subset of vertices $S \subseteq V$ is an independent set if no two vertices in $S$ are adjacent to one another. A maximum independent set $S$ has largest possible size, namely there is no other independent set $S'$ such that $|S'| > |S|$. The independence number of $G$, usually denoted in the literature by $\alpha(G)$, is the size of a maximum independent set of $G$. Independent sets have been subject to extensive research as a consequence of their importance in Ramsey theory \cite{graham1991ramsey}, and due to their relation with many other properties of graphs (such as determination of the chromatic and covering numbers) \cite{bondy1976graph}. Finding a maximum independent set of a given graph was shown to be an NP-complete problem \cite{karp1972reducibility}. 

A typical question is, 
given any integer $a$, what is the smallest possible order $p_a$ of a graph $G$ containing an independent set of size $a$? 
In other words, if $\alpha(G) = a$, how small can $G$ be? For the question to be non-trivial, one needs to impose certain conditions on $G$, i.e., to restrict to a specific class of graphs.

In this paper, we consider the class of polyhedral graphs, namely the graphs that are isomorphic to the $1$-skeleton of a polyhedral solid. By the Rademacher-Steinitz theorem, a graph $G$ is polyhedral if and only if it is planar and $3$-connected \cite{grunbaum2007graphs}. For a polyhedral graph $G$ (or polyhedron for short), regions are also called faces, and each edge is adjacent to exactly two faces. Apart from their combinatorial interest, polyhedra have applications in chemistry where, for instance, they are used to model fullerenes \cite{chen2005spherical,deza2009symmetries}. 

Our first result is the precise value of $p_a$ for each $a$.
\begin{theorem}\label{thm:pa}
	Let $a\geq 1$. The minimal order of a polyhedral graph containing an independent set of size $a$ is given by \[p_a = \left\lceil\frac{3}{2}a + 2\right\rceil.\]
\end{theorem}

Theorem \ref{thm:pa} will be proven in section \ref{sec:i1}. Further, when $a$ is even, we characterise all extremal graphs, i.e., we find all the polyhedra on $\left\lceil\frac{3}{2}a + 2\right\rceil$ vertices with an independent set of size $a$.
\begin{theorem}\label{thm:classifextrem}
	Let $a \geq 4$ be even. Then $G$ is a polyhedron on $p_a$ vertices containing an independent set of size $a$ with minimal number of edges if and only if it is the vertex-face graph of a maximal planar graph with $a$ faces.
\end{theorem}

Theorem \ref{thm:classifextrem} will be proven in section \ref{sec:i2}. The vertex-face graph is also called radial in the literature. For a definition see e.g. \cite[section 2]{archdeacon_construction_1992}.

We now turn to a more general, related concept.
\begin{definition}
	Let $G = (V, E)$ be a graph. A \textbf{k-independent set} is a subset $S \subseteq V$ such that any two vertices in $S$ are at distance greater than $k$ from one another, i.e. \[\forall u, v \in S : \quad d(u,v) > k. \]
	For any integers $k, a \geq 1$, we let $p(k,a)$ be the minimal order of a polyhedron containing a $k$-independent set of size $a$.
\end{definition}

We clearly have $p(1, a)=p_a$. The case $a=1$ is trivial, namely $p(k,1) = 4$ for any $k$, by considering the tetrahedron (the smallest polyhedron). Our next two theorems give the precise value of $p(k, a)$ for each $k,a\geq 2$, distinguishing between the cases of even and odd $k$.
\begin{theorem}\label{thm:pka_even}
	Let $a \geq 2$ and $k$ be an even positive integer. Then \[p(k, a) = a + 3a\cdot\frac{k}{2} = \left(\frac{3}{2}k + 1\right)a.\]
\end{theorem}

Theorem \ref{thm:pka_even} will be proven in section \ref{sec:k}.

\begin{theorem}\label{thm:pka_odd}
	Let $a \geq 2$ and $k$ be an odd positive integer. Then \[p(k, a) = \left\lceil\frac{3}{2}a+2\right\rceil + 3a\cdot \frac{k-1}{2}.\]
\end{theorem}

The result for $k=1$, already stated in Theorem \ref{thm:pa}, will be used in section \ref{sec:k} to prove the case of odd $k\geq 3$.


\paragraph{Related literature.} For every positive $k$, a graph has a $k$-independent set of size $a=2$ if and only if its diameter is at least $k+1$. Therefore, $p(k,2)$ is the order of the smallest polyhedron with diameter $k+1$. Theorems \ref{thm:pka_even} and \ref{thm:pka_odd} are consistent with Klee's results on the diameter of polyhedral graphs  \cite{klee1964diameters}.
\\
The related question of $k$-independent sets for $t$-connected graphs was investigated by Li and Wu \cite{li_k_2021}. In our setting this corresponds to $t=3$ in their article, but we have the extra constraint of planarity. In the case where $k$ is even, we will see that the result in Theorem \ref{thm:pka_even} matches the bound that they obtain, using a similar construction as they used but with fewer edges to keep the graphs planar. On the other hand, in the odd case, the value of $p(k,a)$ obtained in Theorem \ref{thm:pka_odd} is larger than the result in \cite{li_k_2021}, where they obtain, \textit{using our notation}, $p(k, a) = 3a(k-1)/2 + a + 3$, whereas we get $p(k,a) = 3a(k-1)/2 + \lceil (3a/2) + 2\rceil$. This happens since we are considering a smaller class of $3$-connected graphs.



\paragraph{Acknowledgements.}
S.G. worked on this project as partial fulfilment of his master semester project at EPFL, autumn 2022, under the supervision of R.M.
\\
R.M. was supported by Swiss National Science Foundation project 200021\_184927, held by Prof. Maryna Viazovska.

\section{Independent sets}
\subsection{Proof of Theorem \ref{thm:pa}}
\label{sec:i1}
\begin{notation}
We will use the terminology that a vertex is red if it belongs to the independent set $S$ and blue otherwise. 
\end{notation}

In the two following lemmas, we derive the first five values of $p_a$, to perform induction later in the general case. 
\begin{lemma}
For $a \in \{1, 2, 3, 4\}$ the minimal orders of a polyhedral graph containing an independent set of size $a$ are $p_1 = 4, p_2 = 5, p_3 = 7, p_4 = 8$.  
\end{lemma}
\begin{proof}
For $a = 1$, the smallest polyhedron is the tetrahedron (Fig. \ref{fig:G1}) yielding $p_1 = 4$. As the tetrahedron is a complete graph and the only polyhedron on 4 vertices, there is no way to find an independent set of size 2, thus $p_2 \geq 5$. The square pyramid (Fig. \ref{fig:G2}) shows that $p_2 = 5$. The graph $G_3$ (Fig. \ref{fig:G3}) shows that $p_3 \leq 7$. Assume for a contradiction that $p_3 = 6$ and that $G$ is a polyhedral graph of order 6 with three red vertices, say $v_1, v_2, v_3$. Then by $3$-connectivity, $v_1, v_2, v_3$ each have to be adjacent to the three blue vertices $v_4, v_5, v_6$. The subgraph generated by these edges yields a copy of $K(3,3)$, contradicting the planarity of $G$ by Kuratowski's Theorem. Thus $p_3 = 7$. The cube (Fig. \ref{fig:G4}) shows that $p_4 \leq 8$. The exact same argument as for $p_3$ shows that $p_4 > 7$ and thus $p_4 = 8$. 
\end{proof}
\begin{figure}[h]
\centering
    \begin{subfigure}{0.2\textwidth}
        \includegraphics[width=\textwidth]{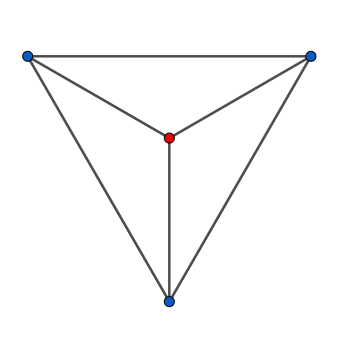}
        \caption{$p = 4,\; a = 1$}
        \label{fig:G1}
    \end{subfigure}
    \hfill
    \begin{subfigure}{0.2\textwidth}
        \includegraphics[width=\textwidth]{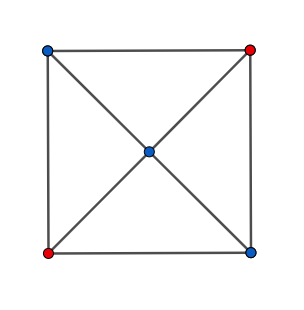}
        \caption{$p = 5,\; a = 2$}
        \label{fig:G2}
    \end{subfigure}
    \hfill
    \begin{subfigure}{0.2\textwidth}
        \includegraphics[width=\textwidth]{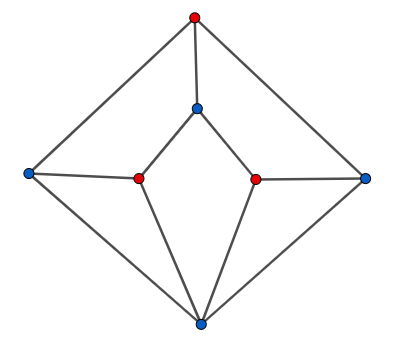}
        \caption{$p = 7,\; a = 3$}
        \label{fig:G3}
    \end{subfigure}
    \hfill
    \begin{subfigure}{0.2\textwidth}
        \includegraphics[width=\textwidth]{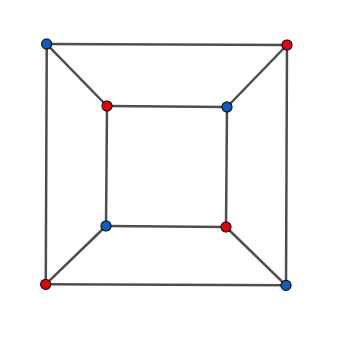}
        \caption{$p = 8,\; a = 4$}
        \label{fig:G4}
    \end{subfigure}
\caption{First four extremal graphs for $p_a$. Vertices in red belong to the independent set.}
\label{fig:Gonetofour}
\end{figure}

\begin{lemma}
It holds that $p_5 = 10$.
\end{lemma}
\begin{proof}
The pseudo double-wheel $PDW_{10}$ shown in Figure \ref{fig:pdw10} shows that $p_5 \leq 10$. Assume for a contradiction that $p_5 < 10$, namely that there exists a graph $G$ of order 9 with an independent set $S \subset V$ of size 5. Label the other vertices with $A, B, C, D$. Then by $3$-connectivity and because there are 4 ways to choose three blue vertices among $V \setminus S$, at least two red vertices $v_1, v_2$ share three common neighbors (say $A, B, C$) and another red vertex $v_3$ shares two common neighbors with them (say $A,C$). By $3$-connectivity, $v_3$ is adjacent to $D$ and $D$ and $B$ are connected by a path that doesn't go through $v_3$. Putting it all together, we get that the subgraph generated by all the edges mentioned above is homeomorphic from $K(3,3)$ (the bipartition being given by $\{A, B, C\}, \{v_1, v_2, v_3\}$), contradicting the planarity of $G$ by Kuratowski's Theorem. We conclude that $p_5 = 10$ as desired. 
\end{proof}
\begin{figure}[h]
    \centering
    \includegraphics[width=0.5\textwidth]{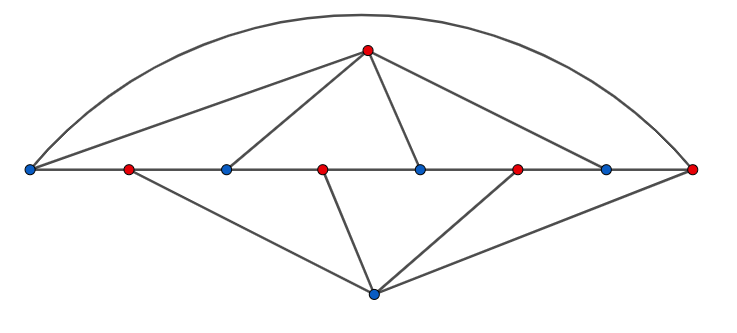}
    \caption{$PDW_{10}$, an extremal graph with $p = 10, \; a = 5$.}
    \label{fig:pdw10}
\end{figure}

At this point one might extrapolate that $p_{a+2} = p_a + 3$ for every $a\geq 1$, and thus the progression is linear. In a moment, we will prove Theorem \ref{thm:pa} confirming this claim.

Before getting there, we introduce a transformation that will play a major role in the proof, and that also gives some intuition as to why this progression occurs. 
Given a planar graph $G$ and a subgraph $H$ on $6$ vertices forming two adjacent quadrilateral faces in $G$, we consider the transformation $\mathcal{P}$ displayed in Figure \ref{fig:transfoP}. It preserves planarity and $3$-connectivity, since it is obtainable via three consecutive applications of the transformation $\mathcal{P}_1$ from \cite[Fig. 3]{brinkmann_generation_2005}. Applying $\mathcal{P}$ transforms the two adjacent quadrilateral faces $H$ on $6$ vertices ino another planar subgraph $H'$ on $9$ vertices with 5 quadrilateral regions.

The subgraph $K$ of $H'$ formed by the two rightmost faces in Figure \ref{fig:transfoP} is isomorphic to $H$ and we can therefore apply $\mathcal{P}$ to $K$. Proceeding like that, we see that if we can apply the transformation $\mathcal{P}$ once to a graph, we can apply it as many times as we want.
\begin{figure}[h]
    \centering
    \includegraphics[width=0.5\textwidth]{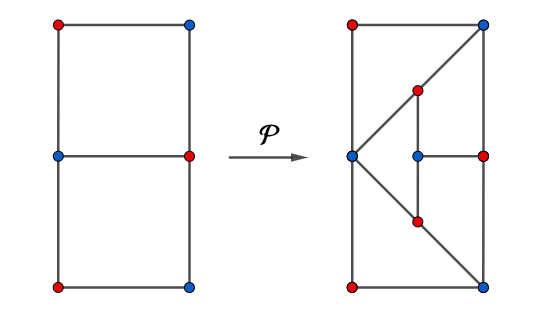}
    \caption{The transformation $\mathcal{P}$}
    \label{fig:transfoP}
\end{figure}


\begin{proof}[Proof of Theorem \ref{thm:pa}]
We start by showing that $p_a \leq \left\lceil(3a/2)+ 2\right\rceil$ by building for any $a$ a polyhedral graph $G$ on $p$ vertices such that $p=\left\lceil(3a/2)+ 2\right\rceil$. Take any integer $a$. For $a \leq 5$, $G$ is given by the results above. For $a \geq 6$, consider the transformation $\mathcal{P}$ described in Figure \ref{fig:transfoP}. If a graph $G$ has $p$ vertices and an independent set of size $a$, $\mathcal{P}(G)$ has $p' = p+3$ vertices and an independent set of size $a' = a + 2$. Thus if $G$ satisfies $p = \left\lceil(3a/2) + 2\right\rceil$ then 
\begin{equation*}
  p' = p + 3 = \left\lceil\frac{3}{2}a + 2\right\rceil + 3 = \left\lceil\frac{3}{2}a + 5\right\rceil = \left\lceil\frac{3}{2}(a+2) + 2\right\rceil = \left\lceil\frac{3}{2}a' + 2\right\rceil  
\end{equation*}
and $\mathcal{P}(G)$ satisfies the same relation. As the cube ($a=4$) and the pseudo double-wheel $PDW_{10}$ ($a=5$) both satisfy the relation and both have at least two adjacent quadrilateral faces, we can apply $\mathcal{P}$ as many times as needed to the cube (if $a$ is even) or to $PDW_{10}$ (if $a$ is odd) to get a polyhedron on $p$ vertices with an independent set of size $a$ such that $p = \left\lceil(3a/2)+ 2\right\rceil$. 

We now show that $p_a \geq \left\lceil(3a/2)+ 2\right\rceil$. Suppose for a contradiction that there is a counterexample $G$ on $p$ vertices with an independent set of size $a$ such that $p < \left\lceil(3a/2)+ 2\right\rceil$. Without loss of generality, we may assume that $G$ does not have any face containing more than two blue vertices. Indeed, if there were one, representing $G$ with that face on the outer boundary and adding a new red vertex outside of $G$ connected to three of these blue vertices yields a new polyhedron $G'$, containing strictly less faces with three blue vertices, and such that $p' = p + 1$ and $a' = a + 1$, thus still satisfying $p' < \left\lceil(3a'/2)+ 2\right\rceil$. Therefore, we may assume that $G$ contains only quadrilateral faces with alternating vertex colours, and triangular faces with one red and two blue vertices. In particular, this implies that all triangular faces must come in pairs sharing a common base formed by a blue-blue edge. Indeed, if the face on the other side of the blue-blue edge of a triangle was a quadrilateral face, it would have to contain three blue vertices. 

Now that our counterexample $G$ is of the desired form, we delete all the blue-blue edges shared by adjacent triangular faces, if any. We are left with a quadrangulation of the sphere, that we denote by $H$, with the same number of vertices $p$ and as many or fewer edges, say $q_H$. Note that while the degrees of some blue vertices may have decreased to $2$ in the process, the red vertices still all have degree greater or equal to $3$ (by $3$-connectivity of $G$). On one hand, as all edges left in the graph are red-blue edges, it holds that \[q_H = \sum_{v \in S}\deg v \geq 3 |S| = 3a.\]
On the other hand, as all faces are quadrilaterals, we have that $2q_H = 4r$, where $r$ is the number of regions. Combining this with Euler's formula, we get that \[q_H = 2p - 4.\]
The two above statements imply that \[2p -4 \geq 3a \quad \text{i.e.} \quad p \geq \frac{3}{2}a + 2.\]

We therefore get that $p$ is an integer such that $(3a/2)+ 2 \leq p < \left\lceil(3a/2)+ 2\right\rceil$, 
contradiction. 
\end{proof}

\subsection{Extremal graphs: proof of Theorem \ref{thm:classifextrem}}
\label{sec:i2}
In the present section, we investigate what the extremal graphs of Theorem \ref{thm:pa} look like in the case of $a$ even and $a$ odd respectively. The structure of the even case yields a characterisation of extremal graphs as vertex-face graphs of a maximal planar graph. 


\begin{proof}[Proof of Theorem \ref{thm:classifextrem}]
Assume that $G$ is a polyhedron on $p = (3a/2)+ 2$ vertices with an independent set of size $a$, i.e. $G$ contains $a$ red and $p-a$ blue vertices. We claim that $G$ is a quadrangulation of the sphere without any separating $4$-cycles and as such the vertex-face graph of some graph $H$ (cf. \cite[section 3.1]{maffucci_self-dual_2022}). Indeed, no face in $G$ can contain more than $2$ blue vertices, for otherwise, we could connect them all to a new vertex, which would yield a polyhedron on $(3a/2)+ 3$ vertices with an independent set of size $a+1$, which contradicts Theorem \ref{thm:pa} (since $a$ is even). As no face has more than $2$ blue vertices, each face has to be either a square with alternating vertex colours or a triangle, with blue-blue edge shared with another triangle. But then if such a pair of triangles occurred, we could remove this blue-blue edge and get a graph with less edges, contradicting the minimality of edges. Thus we indeed have a quadrangulation of the sphere. 

We now show that $G$ has no separating 4-cycles. Suppose by contradiction that $G$ admits a separating 4-cycle $C$ and consider $G_1$ and $G_2$ the non-empty connected components that remain after removing $C$ and let $\overline{G}_i$ be the subgraph generated by  $V(G_i) \cup V(C)$. Say that $G_i$ contains $p_i$ vertices, $a_i$ of which are red, while $C$ contains $4$ vertices, $2$ of which are red. We claim that the $\overline{G}_i$ are $3$-connected. Suppose by contradiction that $\overline{G}_1$ is not $3$-connected. Then it admits a $2$-cut set $\{x, y\}$, with $x$ and $y$ not both in $C$, for otherwise $C$ would be a separating 4-cycle in $\overline{G}_1$ and as such not a face, which it clearly is. They cannot be both inside $G_1$ for otherwise they would be a $2$-cut set in $G$, which is $3$-connected. Finally, it cannot be that $x$ lies on $C$ and $y$ inside $G_1$: indeed if this were the case, take any $u,v$ in different connected components of $G_1 -x - y$. As $G$ is $3$-connected, there exists a $uv$-path $P_{uv}$ in $G - x - y$, which must pass through some vertex $z \in G_2$. The path $P_{uv}$ has thus to cross $C$ in at least two different vertices, say $w_1, w_2$. Then, as only one vertex in $C$ was removed, we can reach $w_2$ from $w_1$ via $C$, which yields a $uv$-path in $G_1 - x - y$. Thus $\overline{G}_1$ is $3$-connected. An analogous argument shows that $\overline{G}_2$ is also $3$-connected. 
\begin{figure}[h]
    \centering
    \includegraphics[width=0.75\textwidth]{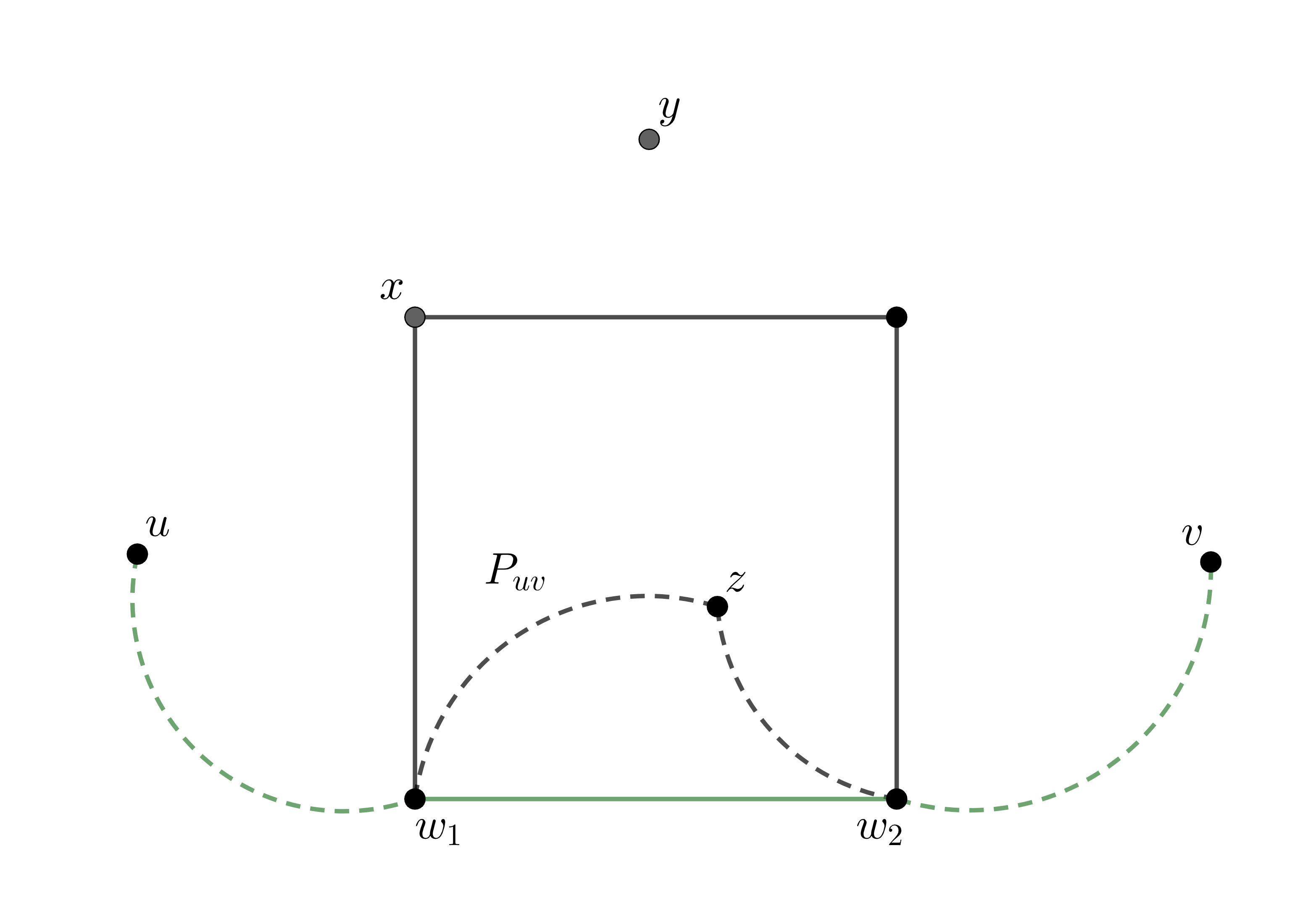}
    \caption{What would happen if $\{x,y\}$ is a $2$-cut set such that $x\in C$ and $y \in G_1$. The dashed line shows the $uv$-path in $G - x - y$ and the green path the resulting $uv$-path in $G_1 - x - y$. Note that curved lines might represent a path formed by more than one edge and that not all edges are displayed.}
    \label{fig:sep4cycle}
\end{figure}

Now, as $\overline{G}_1$ is $3$-connected planar with $a_1$ red vertices it holds by Theorem \ref{thm:pa} that $p_1 + 4 \geq \frac{3}{2}(a_1 + 2) + 2 $. Likewise, $p_2 + 4 \geq \frac{3}{2}(a_2 + 2) + 2$. Summing both expressions, we get 
\[ p_1 + 4 + p_2 + 4 \geq \frac{3}{2}(a_1 + 2 + a_2 + 2) + 4  \quad \text{i.e.} \quad p \geq (3a/2)+ 3,\]
where we used that $p = p_1 + p_2 + 4$ and $a = a_1 + a_2 + 2$. We reach a contradiction and therefore $G$ doesn't have any separating 4-cycle. Thus $G$ is the vertex-face graph of some graph $H$ on $p-a$ vertices containing $a$ regions \cite[Lemma 2.8.2]{mohar_graphs_2001}.

It remains to show that $H$ is maximal planar. By Euler's formula, we have that $q_H = p-2$ and by definition of $p$ we get 
\[p_H = p-a = (3a/2)+ 2 - a = \frac{1}{2}a + 2, \qquad q_H = \frac{3}{2}a + 2 - 2 = \frac{3}{2}a, \qquad r_H = a.\]
From there we can see in two equivalent ways that $H$ is maximal planar. One way is to note that $2q_H = 3r_H$ and as such every region in $H$ is bounded by a triangle. We can also check the relation between vertices and edges, which is satisfied as $3p_H - 6 = (3a/2)+ 6 - 6 = (3a/2)= q_H$. Therefore $H$ is maximal planar, and the first direction holds.

Assume now that $H$ is a $(p', q', r')$ maximal planar graph such that $r' = a$. We consider its vertex-face graph $G$, which is a polyhedron on $p = p' + r'$ vertices \cite[Lemma 2.1]{archdeacon_construction_1992}. As $H$ is a triangulation, $2q' = 3r'$, i.e. $q' = (3a/2)$. Euler's formula yields $p' + r' = q' + 2$, that is $p = (3a/2)+ 2$, which proves the result. 
\end{proof}

We turn to the odd case, where we show that after removing blue-blue edges, the extremal graph is a quadrangulation of the sphere, up to adding one red-blue edge. Similarly to the proof of Theorem \ref{thm:classifextrem}, we use the relation between $p_a$ and $p_{a+1}$ obtained in Theorem \ref{thm:pa}, in order to determine what the different faces must look like. 

\begin{cor}
Let $a \geq 3$ be a fixed odd integer. Let $G$ be any polyhedron on $p_a$ vertices containing an independent set of size $a$ and $H$ the graph obtained by removing all blue-blue edges in $G$. Then either $H$, or $H'$ obtained from $H$ by adding one red-blue edge, is a $2$-connected quadrangulation. 
\end{cor}
\begin{proof}
Since $a$ is odd, we know by Theorem \ref{thm:pa} that there is in $G$ at most one face $F$ containing three blue vertices, as otherwise we would be able to add two red vertices by connecting them to the three blue vertices of two different faces, getting a polyhedral graph on $p_a + 2$ vertices with an independent set of size $a + 2$. By a similar argument, this face cannot contain more than three blue vertices. We are left with four possibilities for what $F$ looks like: a triangle with no red vertices, a square with one red vertex, a pentagon with two red vertices or a hexagon with alternating vertex colours (Figure \ref{fig:Fposs}). Since all faces other than $F$ have at most two blue vertices, we know that there must be a triangle lying on the other side of each blue-blue edge in $F$. Therefore, when removing the blue-blue edges, we are left in any case with a hexagon of alternating vertex colours. Adding an edge from one red vertex to the opposite blue turns the hexagonal face into two quadrilateral faces. By an argument that already appeared in the proof of Theorem \ref{thm:classifextrem} all other faces are quadrilateral or make up a pair of triangles sharing a blue-blue edge. Therefore the graph $H'$ obtained from $G$ by removing all blue-blue edges and potentially adding one edge to a hexagonal face is a $2$-connected quadrangulation.
\end{proof}
\begin{figure}[h]
\centering
    \begin{subfigure}{0.24\textwidth}
        \includegraphics[width=\textwidth]{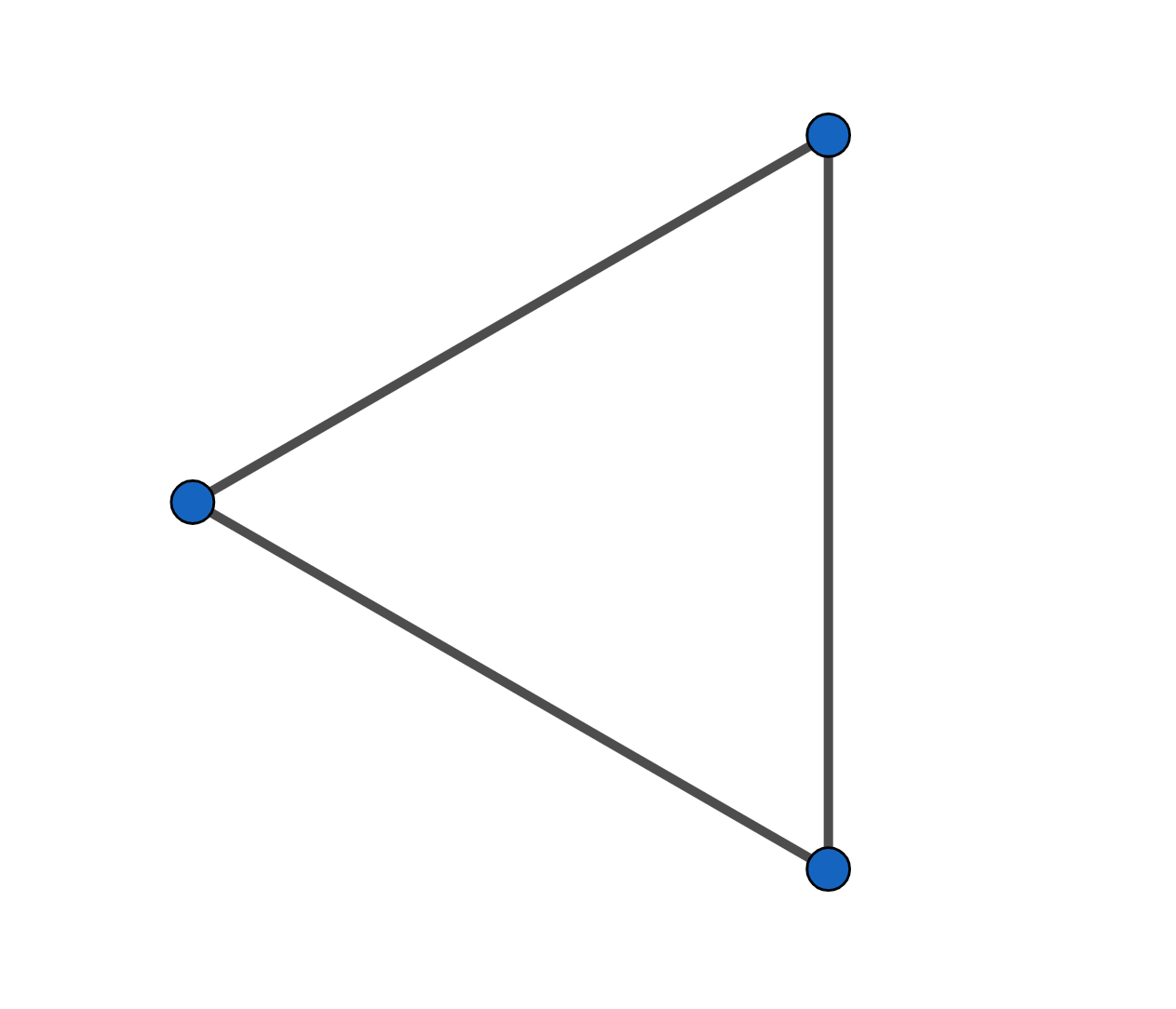}
        \label{fig:3bTri}
    \end{subfigure}
    \hfill
    \begin{subfigure}{0.24\textwidth}
        \includegraphics[width=\textwidth]{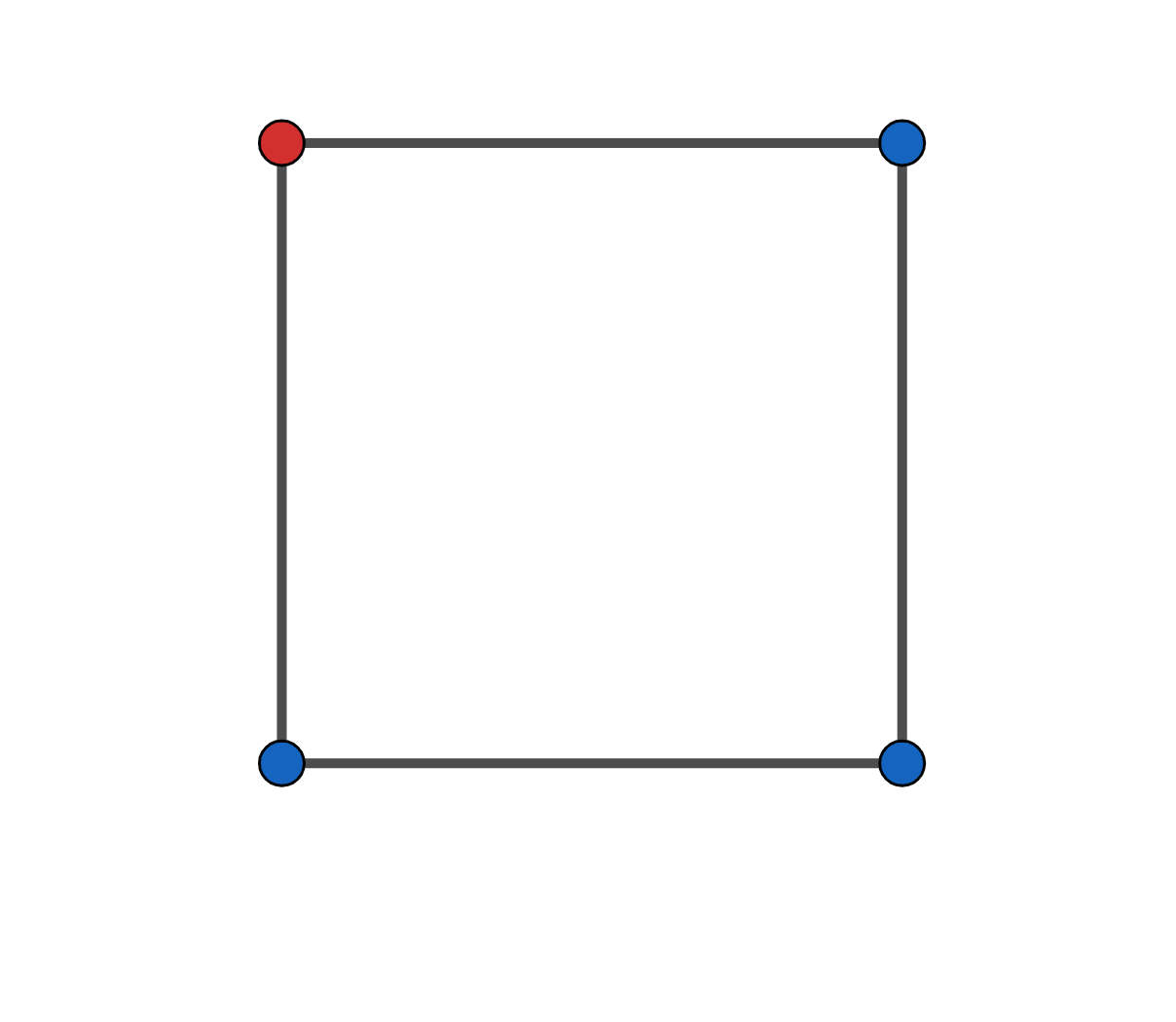}
        \label{fig:3bSqu}
    \end{subfigure}
    \hfill
    \begin{subfigure}{0.24\textwidth}
        \includegraphics[width=\textwidth]{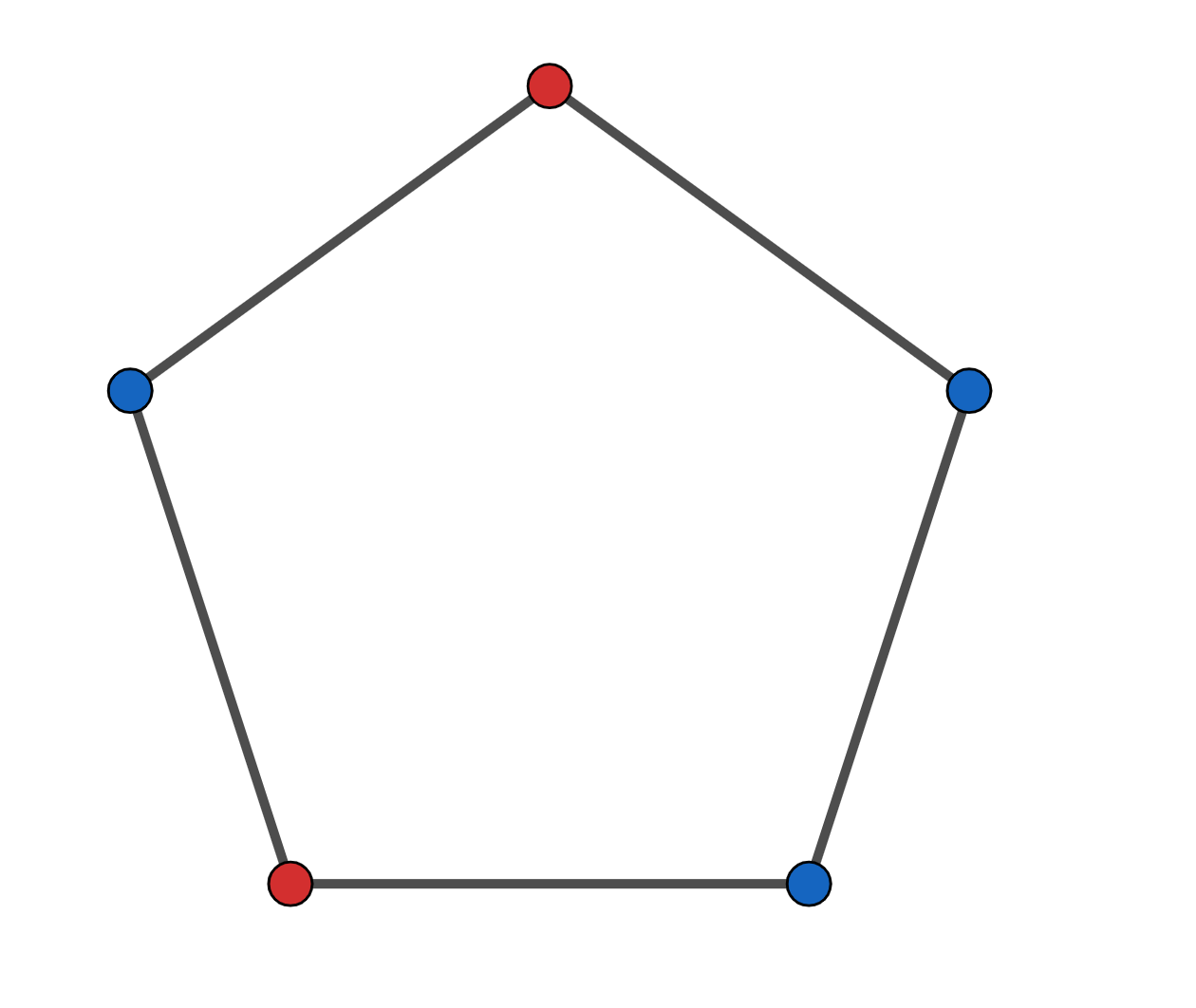}
        \label{fig:3bPen}
    \end{subfigure}
    \hfill
    \begin{subfigure}{0.24\textwidth}
        \includegraphics[width=\textwidth]{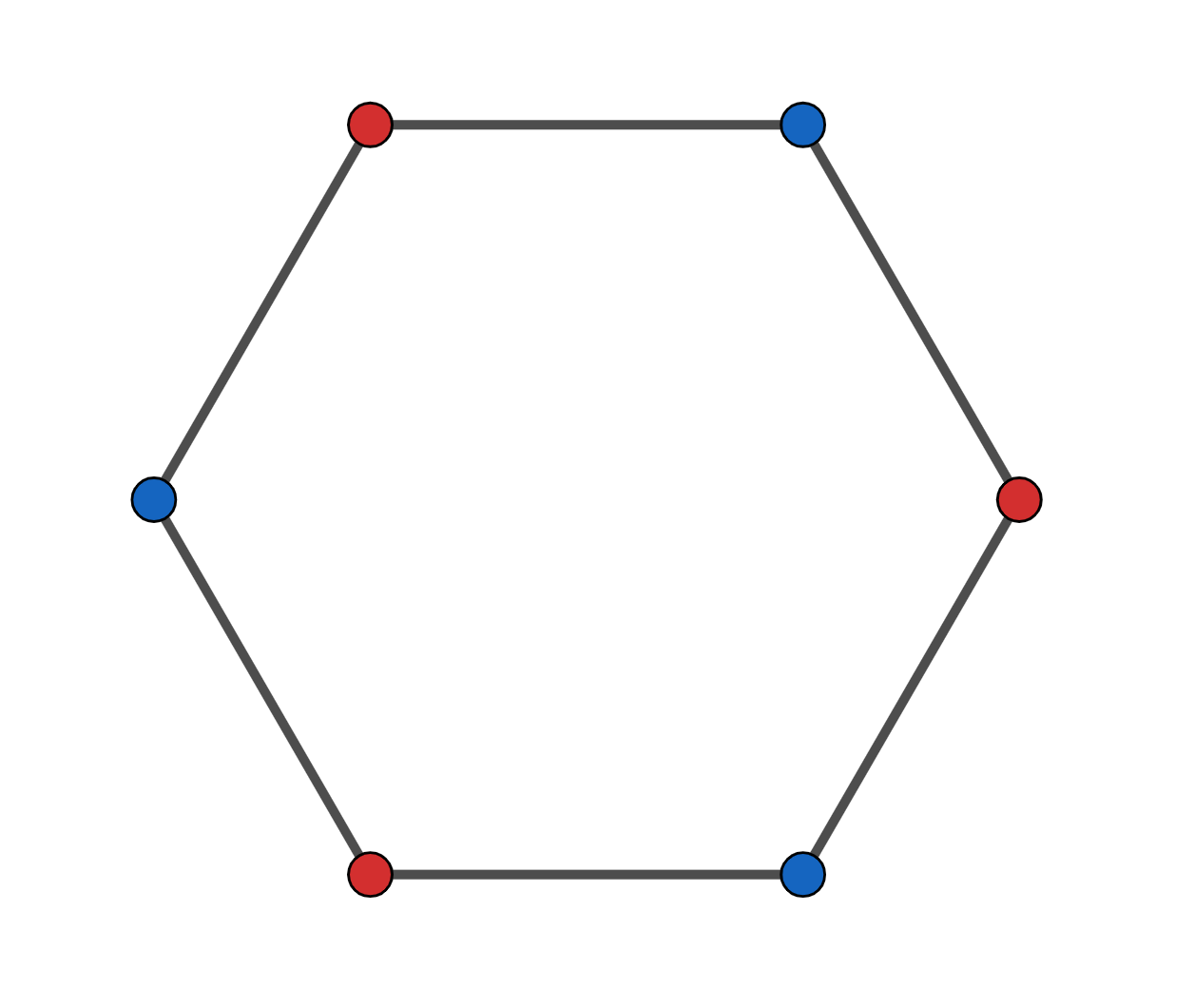}
        \label{fig:3bHex}
    \end{subfigure}
    \hfill
    \begin{subfigure}{0.24\textwidth}
        \includegraphics[width=\textwidth]{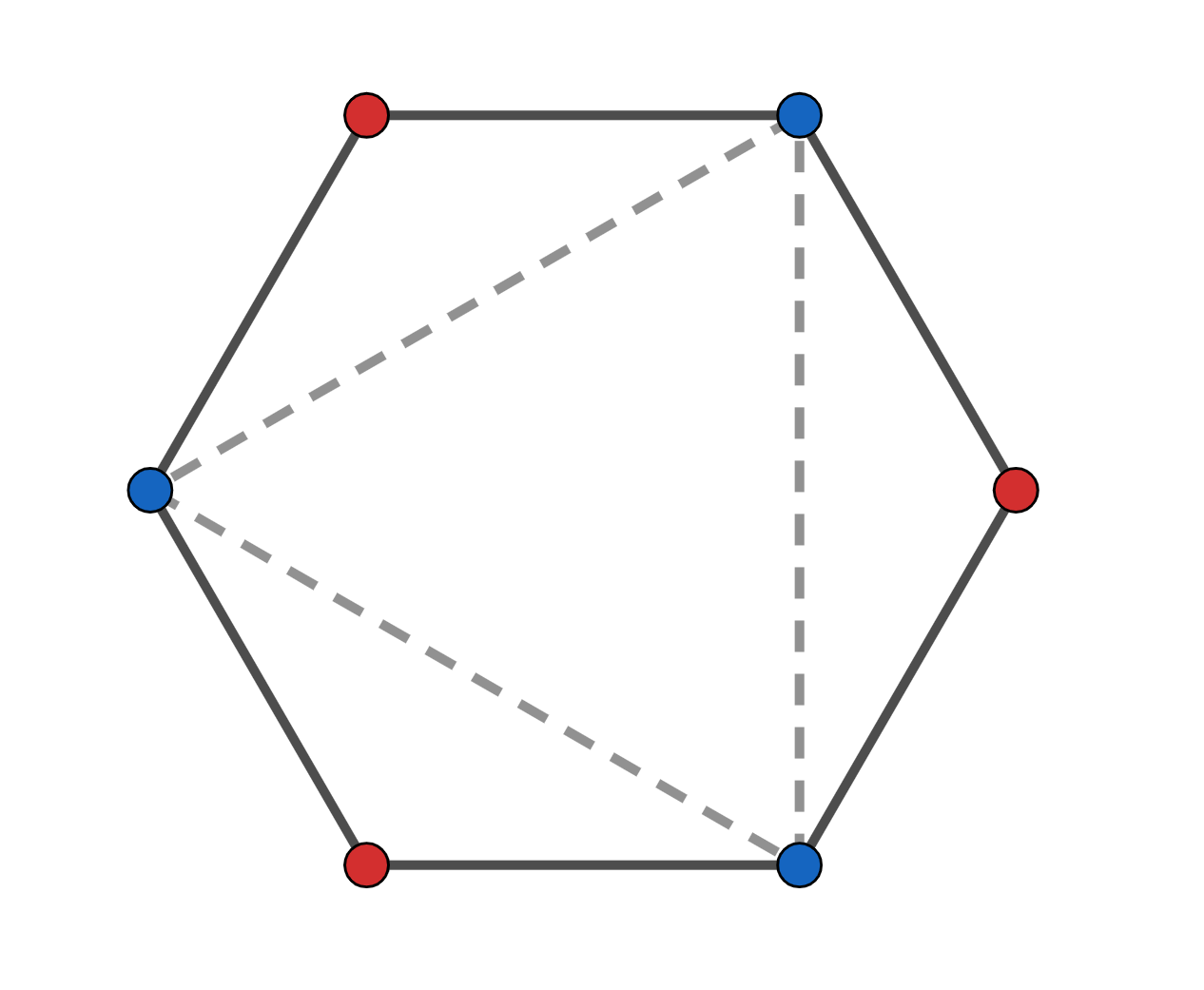}
        \label{fig:TriHex}
    \end{subfigure}
    \hfill
    \begin{subfigure}{0.24\textwidth}
        \includegraphics[width=\textwidth]{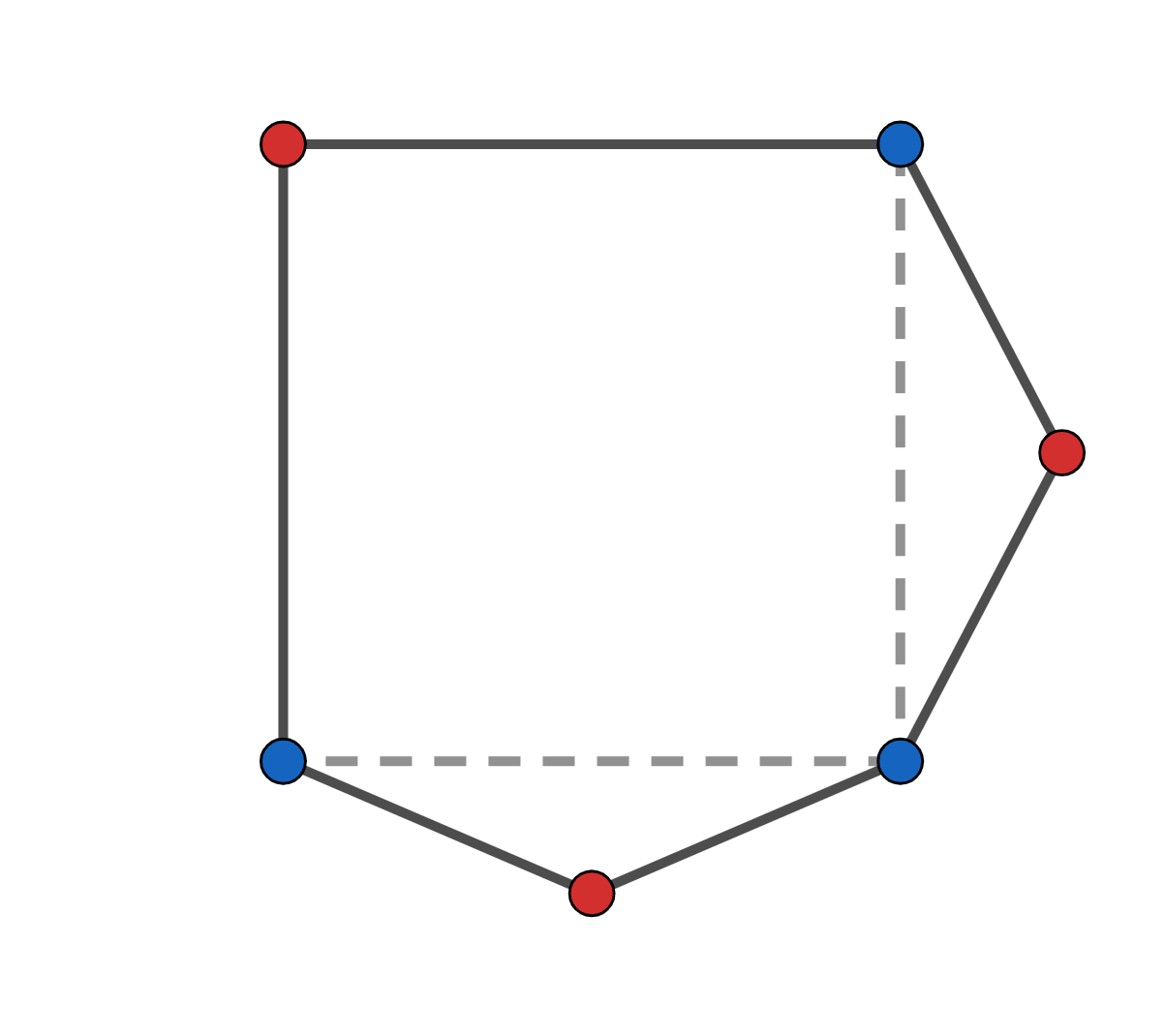}
        \label{fig:SquHex}
    \end{subfigure}
    \hfill
    \begin{subfigure}{0.24\textwidth}
        \includegraphics[width=\textwidth]{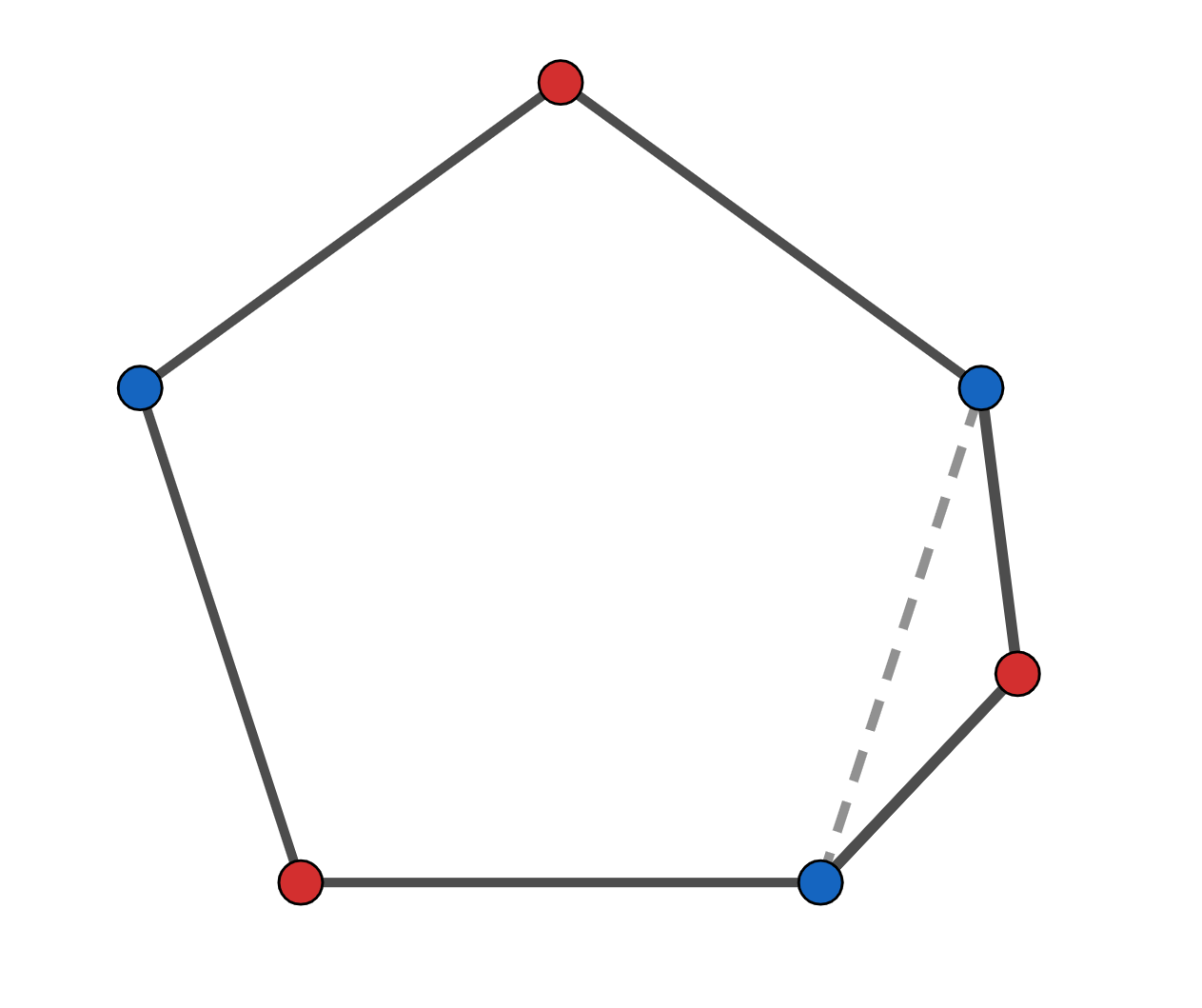}
        \label{fig:PenHex}
    \end{subfigure}
    \hfill
    \begin{subfigure}{0.24\textwidth}
        \includegraphics[width=\textwidth]{Figures/hexagon.png}
        \label{fig:Hex}
    \end{subfigure}
\caption{The four possibilities for the face $F$ with three blue vertices (top row). When we consider what lies on the other side of blue-blue edges, we notice that all cases boil down to the hexagonal case after removing the blue-blue edges (bottom row).}
\label{fig:Fposs}
\end{figure}

We now record a property of the extremal graphs that will be useful in the next section.
\begin{remark}
\label{prop:deg3}
For any $a$, there exists an extremal graph for $p_a$ such that every red vertex has degree 3. Indeed, this holds for $a \leq 4$ by the examples in Figure \ref{fig:Gonetofour}. For $a \geq 5$, we get the desired graph by repeated application of the transformation $\mathcal{P}$ to either $G_3$ or $G_4$ depending on the parity of $a$. We may apply the transformation since both graphs contain two adjacent squares. Further, $\mathcal{P}$ inserts two red vertices of degree $3$ while keeping the degree of the other red vertices unchanged. Therefore the graph obtained satisfies $\deg(v) = 3$ for any $v \in S$ as desired. 
\end{remark}

\section{k-Independent sets}
\label{sec:k}
In this section, we turn our attention to the more general notion of $k$-independent sets, that is a subset $S$ of vertices such that the mutual distance between any pair is always strictly larger than $k$. The case $k = 1$ is precisely the notion of independence investigated in the previous section. We will focus on finding the smallest order $p$ of a polyhedral graph containing a $k$-independent set of size $a$. After treating explicitly the case $k = 2$, we will use it along with the results proven in the previous section and an inductive argument to find a closed formula for $p$ depending on $k$ and $a$.

We start with the case $k=2$. The $3$-connectivity of polyhedral graphs readily yields a lower-bound for $p(2,a)$, which can be matched by building a family of graphs with the desired properties. 
\begin{lemma}\label{thm:p2a}
For every $a \geq 1$, we have $p(2, a) = 4a$.
\end{lemma}
\begin{proof}
For $a = 1$, we have already noted that $p(2, 1) = 4 = 4a$. Now in general, it holds that $p(2, a) \geq 4a$ as each of the $a$ vertices in $S$ has to be adjacent to at least $3$ vertices by $3$-connectivity, and those vertices cannot be adjacent to any of the other vertices in $S$. In other words, denoting by $N_i$ the set of vertices at distance $1$ from $v_i \in S$, we get that all the $N_i$ are disjoint. Therefore,
\[
|V| \geq |S| + \sum_{i=1}^a|N_i| \geq a + 3a = 4a.
\]

To match this bound, we consider a family of graphs $G_a$: we start by taking $a$ copies of $K_4$, labelling the four vertices in copy $i$ by $A_i$, $B_i$, $C_i$ and $D_i$. We connect the copies together as follows: for any $i \in \{1, \dots, a\}$, add an edge between $B_i$ and $B_{i+1}$, and between $C_i$ and $D_{i+1}$ (where the non-existing index $a+1$ is to be understood as $1$) -- see Figure \ref{fig:p2atwoconstr}. The graph $G_a$ obtained is clearly planar and is $3$-connected since one can get from any vertex $u$ to any other vertex $v$ by two disjoint paths using the vertices $C_i$, $D_i$ or by a third disjoint path using the vertices $B_i$. Furthermore, the vertices $A_i$ form a $2$-independent set of size $a$, since they are at distance at least $3$ from one another. Therefore the $G_a$ are indeed polyhedral graphs on $4a$ vertices with a $2$-independent set of size $a$ and we conclude that $p(2, a) = 4a$.

\end{proof}
\begin{figure}[h]
\centering
    \begin{subfigure}{0.275\textwidth}
    \includegraphics[width=\textwidth]{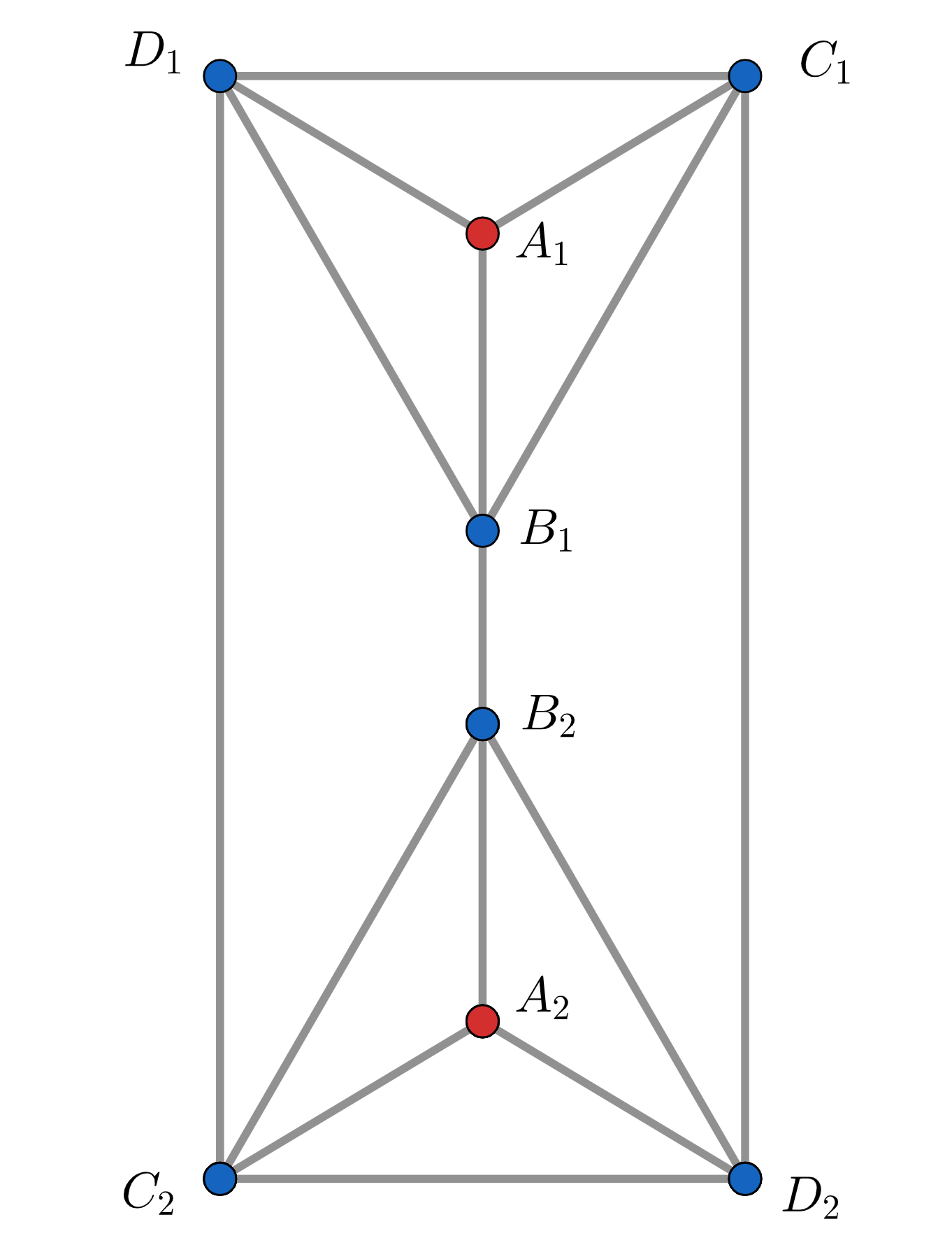}
        \label{fig:a=2}
        \caption{The graph $G_a$ for $a=2$.}
    \end{subfigure}
    \hfill
    \begin{subfigure}{0.45\textwidth}
    \includegraphics[width=\textwidth]{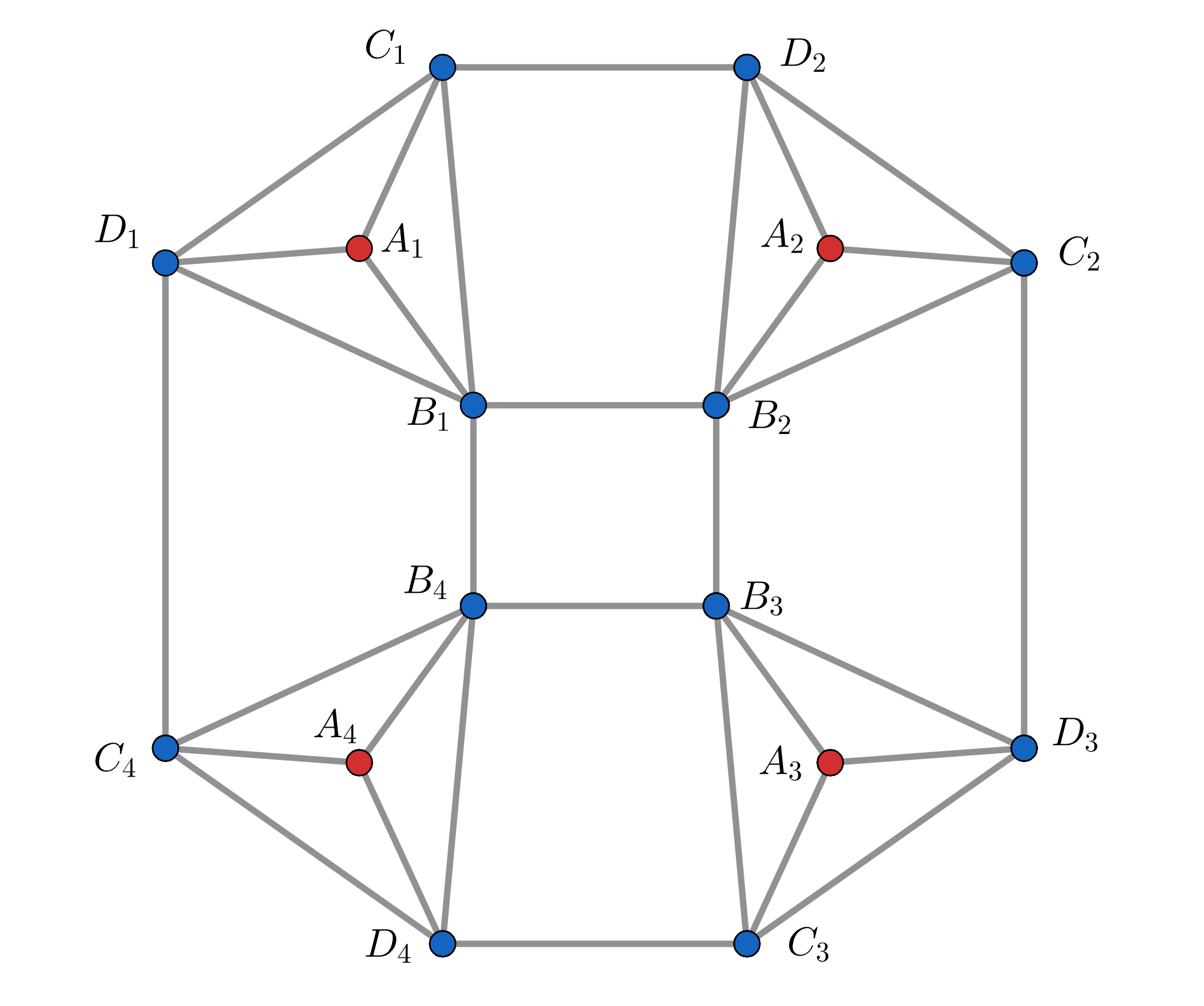}
    \caption{The graph $G_a$ for $a = 4$.}
    \label{fig:p2a}
    \end{subfigure}
\caption{Two examples from the family of graphs $G_a$ used in the proof of Theorem \ref{thm:p2a}. }
\label{fig:p2atwoconstr}
\end{figure}

We can now use this result to find the value of $p(k,a)$ for any even $k$. We show a lower bound for $p(k,a)$ using the same ideas as in the lower bound for Theorem \ref{thm:p2a}, and we then modify the family of graphs $G_a$ by applying the transformation $\mathcal{Q}$ displayed on Figure \ref{fig:transfoQ} - a transformation that has been used for instance by Klee in \cite{klee1976classification}.
\begin{proof}[Proof of Theorem \ref{thm:pka_even}]
We first show that $p(k,a) \geq \left(\frac{3}{2}k + 1\right)a$. Let $S = \{x_1, \dots, x_a\}$ be the $k$-independent set and for each $i \in \{1, \dots, a\}$ let $V_{i,n} := \left\{v \in V : d(x_i, v) = n\right\}$, where $n = 1, \dots, \frac{k}{2}$. Since $d(x_i, x_j) > k$, the sets above are all mutually disjoint. By $3$-connectivity each one contains at least $3$ vertices. We therefore get \[p \geq |S| + \left|\bigcup_{\substack{i = 1, \dots, a \\ j = 1, \dots, k/2}}V_{i,j}\right| \geq a + 3a\cdot\frac{k}{2}.\]

Now, we show that there exists a family of graphs with a $k$-independent set of size $a$ on $a + 3a\cdot\frac{k}{2}$ vertices. We start from the graph $G_a$ built in the proof of Theorem \ref{thm:p2a} and apply $\frac{k}{2} - 1$ times the transformation $\mathcal{Q}$ to each of its $a$ red vertices, which we are allowed to do since each one of these red vertices has degree exactly $3$. Each transformation adds $3$ vertices and increases the distance between red vertices by $2$. Thus starting from $G_a$ on $4a$ vertices where red vertices are at distance greater than $2$, we get a polyhedral graph on $4a + 3a(\frac{k}{2}-1) = a + 3a\frac{k}{2}$ vertices where red vertices are at distance greater than $2 + 2(\frac{k}{2}-1) = k$ from each other, showing that $p(k, a) \leq a + 3a\frac{k}{2}$.
\end{proof}
\begin{figure}[h]
    \centering
    \includegraphics[width=0.5\textwidth]{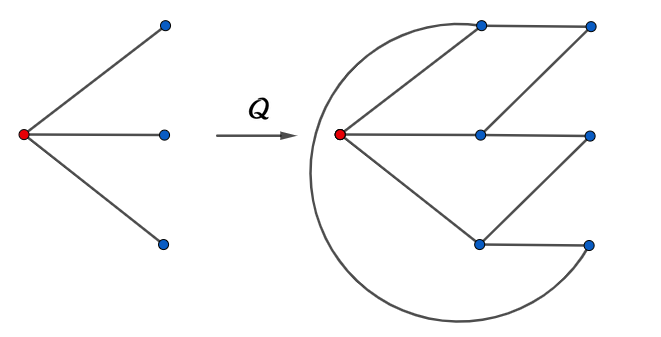}
    \caption{The transformation $\mathcal{Q}$}
    \label{fig:transfoQ}
\end{figure}

The only case remaining is when $k$ is odd. Once again we show a lower bound for $p(k, a)$ by using the $3$-connectivity of polyhedral graphs, but we use a different approach for the matching upper bound. This time we show by contradiction that if there were a smaller graph, we would be able to construct a polyhedron with an independent set of size $a$ on less vertices than what Theorem \ref{thm:pa} allows. 
\begin{proof}[Proof of Theorem \ref{thm:pka_odd}]
We start with an extremal graph $G$ for $p(1,a)$ whose red vertices all have degree $3$, which exists by Remark \ref{prop:deg3}. Using the same reasoning as in Theorem \ref{thm:pka_even} we may apply the transformation $\mathcal{Q}$ to each red vertex $\lfloor\frac{k}{2}\rfloor$ times to get a polyhedral graph with a $k$-independent set of size $a$ on $\left\lceil(3a/2)+2\right\rceil + 3a \frac{k-1}{2}$ vertices, and thus $p(k,a) \leq \left\lceil(3a/2)+2\right\rceil + 3a \frac{k-1}{2}$.

Assume for a contradiction that there exists a polyhedral graph $G$ on $p < \left\lceil(3a/2)+2\right\rceil + 3a \frac{k-1}{2}$ vertices with a $k$-independent set $S = \{x_1, \dots, x_a\}$ of size $a$. For each $i = 1, \dots, a$ and $n = 1, \dots, \frac{k-1}{2}$ define the set $V_{i, n} = \{v \in V : d(x_i, v) = n\}$ and let $Y$ be the set of remaining vertices. By $3$-connectivity and $k$-independence, each of the $V_{i, n}$ are disjoint with cardinality at least $3$. We therefore have that
\[
|Y| \leq p - (a + 3a\cdot\frac{k-1}{2}) < \left\lceil(a/2)+2\right\rceil.
\]
Now consider the graph $H$ obtained from $G$ as follows: $H$ only contains the vertices from $S$ and $Y$, and an edge is added between $x_i \in S, y_j \in Y$ iff $\exists u \in V_{i,(k-1)/2}$ such that $uy_j \in E(G)$. In other words, we collapse all the paths of length $\frac{k+1}{2}$ from $S$ to $Y$ into single edges. The resulting graph is still polyhedral and admits an independent set of size $a$, yet there are only $|S| + |Y| < \left\lceil(3a/2)+2\right\rceil$ vertices, contradicting Theorem \ref{thm:pa}. 
\end{proof}

\bibliography{biblio.bib}
\bibliographystyle{siam}  
\end{document}

%% file: Preamble.tex
\usepackage{amssymb}
\usepackage{amsmath}
\usepackage{amsthm}
\usepackage{amsfonts}
\usepackage{dsfont}
\usepackage{graphicx}
\usepackage{tikz}
\usepackage{bbm}
\usepackage{subcaption}
\usepackage[boxruled]{algorithm2e}
\usepackage{mathtools}
\usepackage{lipsum}
\usepackage[title,titletoc]{appendix}
\usepackage{booktabs}
\usepackage{here}
\usepackage[]{hyperref}
\usepackage[square,sort,comma,numbers]{natbib}

\renewcommand{\phi}{\varphi}

\newtheorem{theorem}{Theorem}[section]
\newtheorem{lemma}[theorem]{Lemma}
\newtheorem{proposition}[theorem]{Proposition}
\newtheorem{cor}[theorem]{Corollary}
\theoremstyle{remark}
\newtheorem{remark}[theorem]{Remark}
\theoremstyle{definition}
\newtheorem{definition}[theorem]{Definition}
\theoremstyle{definition}
\newtheorem{notation}[theorem]{Notation}